\documentclass{amsart}
\usepackage{amsmath}

\usepackage{latexsym, amssymb, amsthm}
\usepackage{enumerate}
\usepackage [all]{xy}

\title{A DNC function that computes no effectively bi-immune set}

\keywords{Computability Theory, Recursion Theory, Degree Theory, DNC, DNR}
\subjclass[2010]{}

\author{Achilles A. Beros}

\address{Laboratoire LINA UMR CNRS 6241\\
UFR de Sciences et Techniques\\
2 rue de la Houssini\`{e}re\\
BP 92208\\
44322 Nantes Cedex 03\\
France
}

\email{achilles.beros@univ-nantes.fr}

\newtheorem{Theorem}{Theorem}[section]

\newtheorem{Lemma}[Theorem]{Lemma}
\newtheorem{Corollary}[Theorem]{Corollary}

\theoremstyle{definition}
\newtheorem{Definition}[Theorem]{Definition}

\newcommand{\upto} {{\upharpoonright}}
\newcommand{\avoid} {\mbox{Avoid}}
\newcommand{\terminal} {\mbox{Terminal}}

\newcommand{\roots} {\mbox{Roots}}
\newcommand{\status} {\mbox{Status}}

\newcommand{\concat} {\hat{\ }}

\begin{document}

\begin{abstract}
In \emph{Diagonally Non-Computable Functions and Bi-Immunity} \cite{joLe}, Carl Jockusch and Andrew Lewis-Pye proved that every DNC function computes a bi-immune set.  They asked whether every DNC function computes an effectively bi-immune set.  We construct a DNC function that computes no effectively bi-immune set, thereby answering their question in the negative.
\end{abstract}

\maketitle

\section{Introduction}

In this paper, we take up the question of the relative computational strengths of diagonally non-computable (DNC) functions and effectively bi-immune sets.  Heuristically, a DNC function is a total function that diagonalizes against every partial computable function.  An immune set is one which, in some sense, avoids the infinite computably enumerable ($c.e.$) sets.  Of specific interest to our investigation are immunity, effective immunity, bi-immunity and effective bi-immunity.  Precise definitions of these terms are given in Definition \ref{immunity-def}.  In 1989, Jockusch showed that every DNC function computes an effectively immune set \cite{degFNFP}.  Recently, Jockusch and Lewis  proved that every DNC function computes a bi-immune set \cite{joLe}.\par

We shall show that this pattern does not persist when we examine effective bi-immunity.  This fact is the substance of our main theorem, which we state without further ado.

\begin{Theorem}\label{mainResult}
There is a diagonally non-computable function that computes no effectively bi-immune set.
\end{Theorem}

%%%%%%%%%%%%%%%%%%%%%%%%%%%%%%%%%%%%%%%%%%%%%%

\section{Preliminaries}\label{preliminaries}

We fix enumerations, $\{ W_e \}_{e \in \mathbb N}$ and $\{ \phi_e \}_{e \in \mathbb N}$, of the computably enumerable sets and the partial computable functions, respectively.  Typically, $e$ will be used to denote a code for either a program or an oracle program, depending on the context.  Lower case Greek letters will denote strings of natural numbers.  We regard a string as a function with domain an initial segment of $\mathbb N$.  As such, we apply standard function notation when manipulating strings.  We denote by $\sigma \prec \tau$ that $\sigma \neq \tau$ and $\tau \upto |\sigma| = \sigma$.  The expression $\sigma\concat\tau$ indicates the extension of $\sigma$ by $\tau$.  In general, we use $T$ to denote an infinite tree and $F$ to denote a finite tree.  The exception to this is during the construction when we need to make distinguished reference to certain particular finite trees, $T_{s}$, which are approximations of the infinite tree $T$ we are constructing.  Given $T$, a tree, $[T]$ denotes the set of infinite branches through $T$.  All trees under consideration are finite branching.  We shall use the notation $\overline{A}$ to indicate the complement of a set $A$.  % If $T$ is finite, we denote by $[T]$ the set of terminal nodes of $T$.\par

% When considering a fixed program with variable finite oracles, we will only be interested in the ``positive" information provided by the oracle.  To be precise, we will view the program as computing a set rather than as enumerating the graph of a function -- the set being the domain of the function whose graph is enumerated by the program.  With this in mind, we make the following definition.

\begin{Definition}
Let $e,s \in \mathbb N$ and $\sigma$ be a string.  We define $W^{\sigma}_{e,s}$ to be the set of natural numbers, $x$, such that, using an oracle which extends $\sigma$ and with oracle usage less than or equal to $|\sigma|$, the computation of the oracle program $e$ terminates on input $x$ within $s$ computation stages.  We define $W^{\sigma}_e = \bigcup_{s\in \mathbb N} W^{\sigma}_{e,s}$.
\end{Definition}

Observe that $W^{\sigma}_{e,s}$ is a finite set and $\overline{W}^{\sigma}_{e,s}$ is computable.\par

We now present formal definitions of immunity and diagonal non-computability.

\begin{Definition}[\cite{soare}]\label{immunity-def}  Let $A$ be an infinite set.
\begin{enumerate}
\item $A$ is \emph{immune} if $A$ contains no infinite $c.e.$ set.

\item $A$ is \emph{bi-immune} if both $A$ and $\overline{A}$ are immune.

\item $A$ is \emph{effectively immune} if there is a computable function, $f$, such that $W_e \subseteq A \rightarrow |W_e|<f(e)$.

\item $A$ is \emph{effectively bi-immune} if both $A$ and $\overline{A}$ are effectively immune.
\end{enumerate}
\end{Definition}

Suppose the computable functions $f$ and $g$ witness effective immunity for a set $A$ and its complement, respectively. Then the function $h(x) = \max\{f(x),g(x)\}$ is computable and witnesses effective immunity for both $A$ and $\overline{A}$.

\begin{Definition}
A function, $f$, is said to be \emph{diagonally non-computable} if $f(e) \neq \phi_e(e)$ whenever $\phi_e(e)$ converges.
\end{Definition}

%%%%%%%%%%%%%%%%%%%%%%%%%%%%%%%%%%%%%%%%%%%%%%

\section{Technical Definitions and a Lemma}

The following technical definitions, notation and lemma are needed for the proof of Theorem \ref{mainResult}.

\begin{Definition}  We establish the following notation.
\begin{enumerate}
\item Given a string $\sigma$ with $|\sigma|>0$, let $\sigma^- = \sigma \upto (|\sigma|-1)$.

\item We fix a one-to-one computable enumeration, $\mathfrak F = \{F_k\}_{k \in \mathbb N}$, of all finite trees contained in $\mathbb N^{<\omega}$.  We say that $F_i <_{\mathfrak F} F_j$ if $i < j$.

\item Given strings $\gamma, \sigma$ and $\tau$ we say that $\sigma$ and $\tau$ are \emph{siblings below $\gamma$} if $\sigma \upto |\gamma| = \tau \upto |\gamma| = \gamma$, $|\sigma| = |\tau|$ and $\sigma(|\gamma|) \neq \tau(|\gamma|)$.  If $\sigma^- = \tau^- = \gamma$, then we say that $\sigma$ and $\tau$ are \emph{siblings immediately below $\gamma$}.

%\item For strings $\sigma$ and $\tau$, we say that $\tau$ is a \emph{descendant} of $\sigma$ if $\sigma \prec \tau$.

\item A node of a tree is a \emph{branching node} if there are at least two members of the tree that are siblings below that node.  If every node at level $\ell$ is a branching node, then level $\ell$ is a \emph{branching level}.

\item We call a tree \emph{full-branching} if every level of the tree that contains a non-terminal node is a branching level.  A finite tree is \emph{full-branching below $\sigma$ of depth $k$} if every level containing a non-terminal node after the $|\sigma|^{th}$ level is a branching level, every node is a prefix or extension of $\sigma$ and every terminal node has length $|\sigma|+k$.  We will call a tree \emph{full-$y$-branching} if it is full-branching and $y$-branching at every node.

%\item For $B \subseteq \mathbb N$, we say that $T$ has \emph{$B$-branching below $n$} if every $x \in B\cap [n, \infty)$ immediately follows a branching level.  If $T$ is $B$-branching below $n$ and $B\cap [n, \infty)\cap \mbox{dom}(f)$ has cardinality $k$ for some $f \in T$, then we say that $T$ has \emph{$B$-branching below $n$ with depth $k$}.

%\item For $B \subseteq \mathbb N$ (either finite or infinite), we say that a binary branching tree $R$ is a \emph{$B$-tree below $\sigma$} if $x-1$ is a braching level for every $x \in B\cap [|\sigma|, \infty)$, there are no other branching levels and every node is either a prefix or extension of $\sigma$.  If $R$ has exactly $k$ branching levels , then we say $R$ has \emph{depth $k$}

%\item We say that a finite $B$-branching tree, $T$, is \emph{bounded by $s$} if $\sigma(n) \leq s$ for every terminal node $\sigma \in T$ and there are at most $s$ branching levels.

%\item For $B \subseteq \mathbb N$ and $n \in \mathbb N$, we define $B[n] = \min \{x: x\in B \wedge x>n+1\} - 1$.  In other words, $B[n]$ is the least branching node greater than $n$ in a B-tree.

\item Given a finite collection of sets, $\mathfrak W$, we call $S$ a \emph{selection of $\mathfrak W$} if $S$ contains exactly one element from each member of $\mathfrak W$.

%\item Let $K$ denote the halting problem, $\{ x : \phi_x(x)\downarrow \}$.
\end{enumerate}
\end{Definition}

The next two definitions characterize certain behavior of finite initial segments of oracles.

\begin{Definition}  We define the terms \emph{accept} and \emph{preserve} for trees and strings.
\begin{enumerate}
\item Let $\sigma$ be a string, $e\in\mathbb N$ and $S$ be a finite subset of $\mathbb N$.  If $S \subseteq W_e^{\sigma}$, we say that $\sigma$ \emph{accepts $S$ for $e$}.  If $S \subseteq \overline{W}_e^{\sigma}$, we say that $\sigma$ \emph{preserves $S$ for $e$}.

\item Given a finite tree, $F$, we say that \emph{$F$ accepts (preserves) $S$ for $e$} if every terminal node of $F$ accepts (preserves) $S$ for $e$.
\end{enumerate}
\end{Definition}

It is important to note that if a finite tree, $F$, accepts a finite set, $S$, then every tree containing $F$, each of whose nodes is a prefix or suffix of some node in $F$, also accepts $S$.

\begin{Definition}\label{def-bad}
Let $\mathfrak W$ be a finite collection of finite sets, $\sigma$ a string and $e,n,m \in \mathbb N$.  We define $\sigma$ to be \emph{$n$-bad relative to $\langle \mathfrak W, m, e \rangle$} inductively.

\begin{enumerate}
\item $\sigma$ is \emph{1-bad relative to $\langle \mathfrak W, m, e \rangle$} if $\sigma$ accepts a selection of $\mathfrak W$ for $e$.

\item $\sigma$ is \emph{$(n+1)$-bad relative to $\langle \mathfrak W, m, e \rangle$} if $\sigma$ has $m$ extensions of length $|\sigma|+1$ that are $n$-bad relative to $\langle \mathfrak W, m, e \rangle$.
\end{enumerate}
\end{Definition}

When it is clear from context, we omit the statement of relativization and merely say that a string is $n$-bad.  Also, when computational bounds are required, we say that a string is $n$-bad within $s$ computation stages.  Observe that $\ell$-bad implies $n$-bad for $\ell \leq n$.\par

Figure 1, below, illustrates the relationship between different levels of badness.  The empty circles indicate nodes that are 1-bad.

\[\begin{xy}
(0,0)*+{}*\cir<1pt>{}*\frm{*}="a";
(5,0)*+{\mbox{\tiny{4-bad}}}="root";
(-36,-16)*+{}*\cir<1pt>{}*\frm{*}="b0";
(-31,-16)*+{\mbox{\tiny{3-bad}}}="3b0";
(0,-16)*+{}*\cir<1pt>{}*\frm{*}="b1";
(5,-16)*+{\mbox{\tiny{2-bad}}}="2b1";
(36,-16)*+{}*\cir<1pt>{}*\frm{}="b2";
(41,-16)*+{\mbox{\tiny{1-bad}}}="1b2";
(-48,-24)*+{}*\cir<1pt>{}*\frm{*}="c0";
(-43,-24)*+{\mbox{\tiny{2-bad}}}="2c0";
(-36,-24)*+{}*\cir<1pt>{}*\frm{*}="c1";
(-31,-24)*+{\mbox{\tiny{2-bad}}}="2c1";
(-24,-24)*+{}*\cir<1pt>{}*\frm{}="c2";
(-19,-24)*+{\mbox{\tiny{1-bad}}}="1c2";
(-12,-24)*+{}*\cir<1pt>{}*\frm{}="c3";
(-7,-24)*+{\mbox{\tiny{1-bad}}}="1c3";
(0,-24)*+{}*\cir<1pt>{}*\frm{}="c4";
(5,-24)*+{\mbox{\tiny{1-bad}}}="1c4";
(12,-24)*+{}*\cir<1pt>{}*\frm{}="c5";
(17,-24)*+{\mbox{\tiny{1-bad}}}="1c5";
(24,-24)*+{}*\cir<1pt>{}*\frm{}="c6";
(36,-24)*+{}*\cir<1pt>{}*\frm{}="c7";
(48,-24)*+{}*\cir<1pt>{}*\frm{}="c8";
(-52,-28)*+{}*\cir<1pt>{}*\frm{}="d0";
(-48,-28)*+{}*\cir<1pt>{}*\frm{}="d1";
(-44,-28)*+{}*\cir<1pt>{}*\frm{}="d2";
(-40,-28)*+{}*\cir<1pt>{}*\frm{}="d3";
(-36,-28)*+{}*\cir<1pt>{}*\frm{}="d4";
(-32,-28)*+{}*\cir<1pt>{}*\frm{}="d5";
(-28,-28)*+{}*\cir<1pt>{}*\frm{}="d6";
(-24,-28)*+{}*\cir<1pt>{}*\frm{}="d7";
(-20,-28)*+{}*\cir<1pt>{}*\frm{}="d8";
(-16,-28)*+{}*\cir<1pt>{}*\frm{}="d9";
(-12,-28)*+{}*\cir<1pt>{}*\frm{}="d10";
(-8,-28)*+{}*\cir<1pt>{}*\frm{}="d11";
(-4,-28)*+{}*\cir<1pt>{}*\frm{}="d12";
(0,-28)*+{}*\cir<1pt>{}*\frm{}="d13";
(4,-28)*+{}*\cir<1pt>{}*\frm{}="d14";
(8,-28)*+{}*\cir<1pt>{}*\frm{}="d15";
(12,-28)*+{}*\cir<1pt>{}*\frm{}="d16";
(16,-28)*+{}*\cir<1pt>{}*\frm{}="d17";
(20,-28)*+{}*\cir<1pt>{}*\frm{}="d18";
(24,-28)*+{}*\cir<1pt>{}*\frm{}="d19";
(28,-28)*+{}*\cir<1pt>{}*\frm{}="d20";
(32,-28)*+{}*\cir<1pt>{}*\frm{}="d21";
(36,-28)*+{}*\cir<1pt>{}*\frm{}="d22";
(40,-28)*+{}*\cir<1pt>{}*\frm{}="d23";
(44,-28)*+{}*\cir<1pt>{}*\frm{}="d24";
(48,-28)*+{}*\cir<1pt>{}*\frm{}="d25";
(52,-28)*+{}*\cir<1pt>{}*\frm{}="d26";
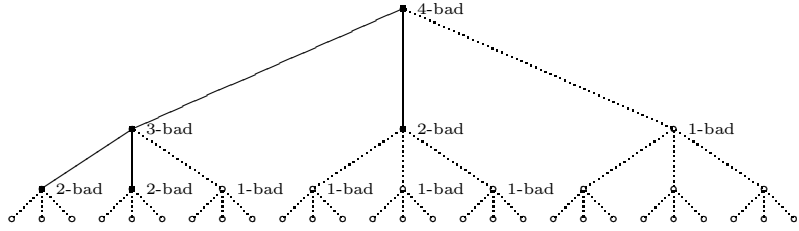
(0,-33)*+{\mbox{\small{Figure 1:  Relative badness for $m=3$}}}="caption";
%(60,0)*+{}="ref0";
%(60,-16)*+{}="ref1";
%(60,-24)*+{}="ref2";
%(60,-28)*+{}="ref3";
"a";"b0" **\dir{-},
"a";"b1" **\dir{-},
"a";"b2" **\dir{.},
"b0";"c0" **\dir{-},
"b0";"c1" **\dir{-},
"b0";"c2" **\dir{.},
"b1";"c3" **\dir{.},
"b1";"c4" **\dir{.},
"b1";"c5" **\dir{.},
"b2";"c6" **\dir{.},
"b2";"c7" **\dir{.},
"b2";"c8" **\dir{.},
"c0";"d0" **\dir{.},
"c0";"d1" **\dir{.},
"c0";"d2" **\dir{.},
"c1";"d3" **\dir{.},
"c1";"d4" **\dir{.},
"c1";"d5" **\dir{.},
"c2";"d6" **\dir{.},
"c2";"d7" **\dir{.},
"c2";"d8" **\dir{.},
"c3";"d9" **\dir{.},
"c3";"d10" **\dir{.},
"c3";"d11" **\dir{.},
"c4";"d12" **\dir{.},
"c4";"d13" **\dir{.},
"c4";"d14" **\dir{.},
"c5";"d15" **\dir{.},
"c5";"d16" **\dir{.},
"c5";"d17" **\dir{.},
"c6";"d18" **\dir{.},
"c6";"d19" **\dir{.},
"c6";"d20" **\dir{.},
"c7";"d21" **\dir{.},
"c7";"d22" **\dir{.},
"c7";"d23" **\dir{.},
"c8";"d24" **\dir{.},
"c8";"d25" **\dir{.},
"c8";"d26" **\dir{.},
%"a";"ref0" **\dir{.},
%"b2";"ref1" **\dir{.},
%"c5";"ref2" **\dir{.},
%"d11";"ref3" **\dir{.}
\end{xy}\]

\smallskip

%%%%%%%%%%%%%%%%%%%%%%%%%%%%%%%%%%%%%%%%%%%%%%

\begin{Lemma}\label{lemma}
Fix $\mathfrak W = \{ A_0, \ldots A_p \}$ a finite collection of finite sets, $e,k,y\in \mathbb N$ and $\sigma$, a string.  One of the following is true:

\begin{enumerate}
\item \label{disjunct1} There is a selection of $\mathfrak W$, $S$, and a finite tree, $F$, such that $F$ is $y$-branching below $\sigma$, every terminal node has length $|\sigma|+k$, and $F$ accepts~$S$.

\item \label{disjunct2} There are fewer than $m = (y-1)(\prod_{j\leq p}|A_j|)+1$ siblings immediately below $\sigma$ that are $i$-bad relative to $\langle \mathfrak W, m, e \rangle$ for some $i \leq k$.
\end{enumerate}

\end{Lemma}

\begin{proof}
We first prove the claim for $y=2$.  Suppose \ref{disjunct2} above is false.  We prove that \ref{disjunct1} is true by induction on the degree of badness of the siblings immediately below $\sigma$.  Let $\tau_1,\ldots,\tau_m$ be siblings immediately below $\sigma$ that are $1$-bad relative to $\langle \mathfrak W, m, e \rangle$.  For $1\leq j\leq m$, we can choose a full-branching tree below $\tau_j$ of depth $k-1$ which accepts a selection of $\mathfrak W$.  Since $\mathfrak W$ has fewer than $m$ distinct selections, by the pigeon-hole principle, two of the trees must accept the same selection.  The union of these two trees is the desired tree, $F$.\par

We now assume that \ref{disjunct1} holds whenever there are $m$ siblings immediately below $\sigma$ that are $n$-bad.  Suppose there are $m$ siblings immediately below $\sigma$ that are $(n+1)$-bad relative to $\langle \mathfrak W, m, e \rangle$; each has $m$ siblings below it that are $n$-bad relative to $\langle \mathfrak W, m, e \rangle$.  By the induction hypothesis, each of the $(n+1)$-bad nodes is extended by a full-branching tree that accepts some selection of $\mathfrak W$.  Since $m$ $(n+1)$-bad nodes extend $\sigma$, the induction step follows by again applying the pigeon-hole principle.\par

To prove the claim for $y>2$, observe that everywhere a pair of nodes was selected, $y$ nodes can be selected, and everywhere a full-branching tree was chosen, a full-$y$-branching tree can be chosen.
\end{proof}

We present one final definition before turning to the proof of the main result.

\begin{Definition}
Given an effective construction, we call a natural number a \emph{diagonalization point} of the construction if it is a code for a c.e.~set whose enumeration can be controlled during the construction.
\end{Definition}

The existence of diagonalization points is guaranteed by the Recursion Theorem.

%%%%%%%%%%%%%%%%%%%%%%%%%%%%%%%%%%%%%%%%%%%%%%

\section{Proof of the Theorem}\label{proof}

With Lemma \ref{lemma} in hand, we proceed to prove the theorem.

\begin{proof}[Proof of Theorem \ref{mainResult}]

The proof uses the priority method.  We shall construct an infinite full-branching tree $T$ such that no member of $[T]$ computes an effectively bi-immune set.  The fact that $T$ will be full branching will guarantee that some member of $[T]$ will be a DNC function, which will prove the claim.

Given a partial computable function $h$ and an oracle program $e$, we consider the requirements $R_{\langle h, e \rangle}$:
$$(\exists x)\Big( h(x)\uparrow \Big) \vee (\forall f \in [T])(\exists a)\Big( h(a) < |W_a| \wedge (W_a \subseteq W_e^f \vee W_a \subseteq \overline{W}_e^f) \Big).$$
% Do I even use P_s?

%We fix a computable enumeration of all pairs, $\langle h, e \rangle$, of partial computable functions and oracle programs, such that every pair appears an infinite number of times in the enumeration.  We denote by $P_s$ the $s^{th}$ pair in the enumeration.

In order to describe the strategy for satisfying $R_{\langle h, e \rangle}$ we must define a number of functions and sets that will be used in the strategy and updated throughout the construction.

%%%%%%%%%%%%%%%%%%%%%%%%%%%%%%%%%%%%%%%%%%%%%%
%																		  %
%		Introducing the recursive definitions of avoid, roots, etc.					  %
%																		  %
%%%%%%%%%%%%%%%%%%%%%%%%%%%%%%%%%%%%%%%%%%%%%%

\begin{itemize}
\item Let $\terminal$ be the set of nodes that will never be extended.

\item For each requirement, $R_{\langle h, e \rangle}$, a collection of strings, $\roots(\langle h, e \rangle)$, will be chosen.  We will ensure that every element of $T$ is either a prefix or an extension of some member of $\roots(\langle h, e \rangle)$, or is a member of $\terminal$.  The elements of $\roots(\langle h, e \rangle)$ are called \emph{roots of $R_{\langle h, e \rangle}$}.  The requirement $R_{\langle h, e \rangle}$ will be satisifed for the entire sub-tree below a root at once.  The requirement will have an outcome for each root.

\item Given a root, $\sigma$, when certain conditions are met we will assign a collection of sets, $\mathfrak W(\sigma)$, to that root.  Exactly how the members of $\mathfrak W(\sigma)$ are chosen is described in step \ref{pickSets} of the strategy.  Depending on the outcome of the requirement for $\sigma$, this collection of sets may include the witness needed to diagonalize against $h$.

\item We will use the values of a function, $\status$, to indicate which roots require protection from injury by lower priority requirements.  In particular, given a root, $\sigma$, $\status(\sigma)=0$ if a witness has not been found in $\mathfrak W(\sigma)$ and protection is required.  If $\status(\sigma) = 1$, then a witness has been found and protection is not required.

\item For each root, $\sigma$, a tree, $F(\sigma)$, which accepts a selection of~$\mathfrak W(\sigma)$, may be chosen.  After $F(\sigma)$ is chosen, $\status(\sigma)$ is set to 1.

\item If $\sigma$ is a root of some requirement and $\status(\sigma)=0$, then we denote by $\avoid(\sigma)$ the set 
$$\{ \beta : ( \exists k, \alpha )( \beta \succeq \alpha \succ \sigma \wedge \alpha \mbox{ is $k$-bad relative to $\langle \mathfrak W(\sigma), m^*(\sigma), e \rangle$} \},$$
\noindent where $m^*(\sigma)$ is defined recursively based on the outcomes of roots that are prefixes of $\sigma$.  The precise definition of $m^*(\sigma)$ will be given in step \ref{pickSets} of the strategy.  If $\status(\sigma) = 1$, then $\avoid(\sigma) = \emptyset$.  $\avoid$ is the restraint function which protects computations by higher priority requirements from the actions of lower priority requirements.  Since $\avoid$ is not computable, at any given stage of the construction we must use a computationally bounded approximation.
\end{itemize}

\textbf{Strategy:}  The strategy is identical for each root, so we assume a fixed $\sigma \in \roots(\langle h, e \rangle)$ and describe the strategy for satisfying $R_{\langle h, e \rangle}$ below $\sigma$.

%%%%%%%%%%%%%%%%%%%%%%%%%%%%%%%%%%%%%%%%%%%%%%
%																		  %
%					STRATEGY FOR R_{\lange h, e \rangle}					  %
%																		  %
%%%%%%%%%%%%%%%%%%%%%%%%%%%%%%%%%%%%%%%%%%%%%%

\begin{enumerate}[1.]

\item \label{notTotal1} First, we select a diagonalization point, $a$, and wait for $h(a)$ to converge. 

\item \label{notTotal2} If $h(a)$ converges, we select further diagonalization points, $b_{0}, \ldots , b_{h(a)}$, and wait for $h$ to converge on the new diagonalization points.  

\item \label{pickSets} If $h$ converges on the diagonalization points $b_{0}, \ldots , b_{h(a)}$, we define $m(\sigma) = \max \{ h(b_{j}) : j \leq h(a) \}$ and
$$m^*(\sigma) = \Big( 1 + \sum_{\tau \prec \sigma} m^*(\tau) \Big)\Big( m(\sigma)^{h(a)+1} \Big)+1,$$
where $\sum_{\tau \prec \sigma} m^*(\tau) = 1$ when $\sigma$ is the empty string.  We let $y = 2+\sum_{\tau \prec \sigma} m^*(\tau)$.  We choose sets $Y_0, \ldots, Y_{h(a)}$, disjoint subsets of $\overline W^{\sigma}_{e}$ of size $m(\sigma)$, and set $W_{b_j} = Y_j$ for $j \leq h(a)$.  We complete step 3 by defining $\mathfrak W(\sigma) = \{ W_{b_0}, \ldots , W_{b_{h(a)}} \}$.

\item \label{stat0} Next, we search for $m^*(\sigma)$ siblings immediately below $\sigma$ that are $k$-bad relative to $\langle \mathfrak W(\sigma), m^*(\sigma), e \rangle$ for some $k \in \mathbb N$.  While searching, $\status(\sigma) = 0$.

\item \label{endOfS} If the search in step \ref{stat0} terminates, by Lemma \ref{lemma} we can choose a finite tree, $F(\sigma)$, of depth $k$ which is $y$-branching below $\sigma$ and accepts some selection of $\mathfrak W(\sigma)$.  We enumerate into $W_{a}$ the selection of $\mathfrak W(\sigma)$ that $F(\sigma)$ accepts, set $\status(\sigma) = 1$ and add to $T$ all nodes of $F(\sigma)$.  For each $\alpha \prec \sigma$ such that $\status(\alpha) = 0$, we add to the set $\terminal$ those nodes of $F(\sigma)$ in $\avoid(\alpha)$.

\end{enumerate}

If the strategy remains at steps \ref{notTotal1}, \ref{notTotal2} or \ref{stat0} cofinitely, we denote the outcome $\infty$.  If the strategy completes step \ref{endOfS} and selects a tree $F(\sigma)=F_{i} \in \mathfrak{F}$ we denote the outcome $w_{i}$.\par

\textbf{Tree of Strategies:}  We define $\Lambda = \{ \infty \} \cup \{ w_{i} : i \in \mathbb N \}$ and order $\Lambda$ so that $\infty$ is the $<_{\Lambda}$-least outcome and $w_i <_{\Lambda} w_j$ if and only if $F_i <_{\mathfrak F} F_j$ (i.e., $i<j$).  We choose $\{f_i\}_{i \in \mathbb N}$, an effective enumeration of the functions with domain contained in $\mathbb N^{<\omega}$ and range in $\Lambda$.  We denote by $\Delta$ the set of all such functions $\{ f_i : i \in \mathbb N\}$.  We say that $f_i<_{\Delta} f_j$ if and only if $\mbox{dom}(f_i) = \mbox{dom}(f_j)$ and, for all $\sigma \in \mbox{dom}(f_i)$, $f_i(\sigma) <_{\Lambda} f_j(\sigma)$.  We construct the tree of strategies, $\mathfrak T$, inductively as a subset of $\Delta^{<\omega}$ so that $\alpha \in \mathfrak T$ is extended by all possible results of the strategy $\alpha$. \par

\textbf{Construction:}  We proceed in stages, constructing a finite tree, $T_s$, at each stage $s$.  Consider an arbitrary stage $s$.  We divide stage $s$ into $s$ sub-stages and at each sub-stage $t$, $0 \leq t < s$, we consider the $t^{th}$ requirement, $R_{\langle h_t,e_t \rangle}$.  For every terminal node, $\sigma\in T_{s}$, not in Terminal, we choose two siblings immediately below $\sigma$, which are not in $\avoid(\tau)$ for any $\tau\prec\sigma$, and add them to $T_{s}$.  This collection of new terminal nodes is the set $\roots(\langle h_t,e_t \rangle)$.

Consider $f\in \Delta$.  If the strategy $\alpha \in \mathfrak T$ acted at sub-stage $t-1$, $\mbox{dom}(f) = \roots(\langle h_t,e_t \rangle)$ and $f$ reflects a current correct outcome of each member of $\roots(\langle h_t,e_t \rangle)$, then $\alpha\concat f$ is eligible to act at sub-stage $t$.  From those $f$, such that $\alpha \concat f$ is eligible to act, we pick the $<_{\Lambda}$-least $f$.  We now act according to the \ref{endOfS} step strategy described above for each root in $\roots(\langle h_t,e_t \rangle)$ using the approximations of $\avoid$ and $\terminal$ computed up to $s$ stages.\par

We define $T=\bigcup_{s\in \mathbb N} T_s$.

\textbf{Verification:}  We shall refer to the outcome of a requirement at one of its roots as the \emph{outcome of the root}.  Let $f$ denote the $\liminf$ of the current strategy.  We prove inductively that for every requirement the set of roots eventually stabilizes.

The claim is obvious for the first requirement as the only root is the empty string and there are no higher priority requirements.  The first requirement is the base case of our induction.\par

Now suppose that $\tau$ is a root and the true outcome of $\tau$ is $\infty$.  From Definition \ref{def-bad} it follows that, for $\alpha \succ \tau$ not in $\avoid(\tau)$, there are fewer than $m^*(\tau)$ siblings immediately below $\alpha$ in $\avoid(\tau)$.  In other words, every root with true outcome $\infty$ can exclude, at every level, no more than $m^*(\tau)$ strings that are not already in a cone to be avoided.  Now consider a permanent root, $\sigma$, that reaches step \ref{endOfS} of the strategy and has outcome $w_i$, with $F(\sigma) = F_i$.  Even if every root $\tau \prec \sigma$ has true outcome $\infty$ and selects the maximum number, $m^*(\tau)$, of cones to be avoided at every level, $F_i \setminus \bigcup_{\tau \prec \sigma} \avoid(\tau)$ will still contain a full-branching tree below $\sigma$ since $F_i$ is $y$-branching for $y = 2+\sum_{\tau \prec \sigma} m^*(\tau)$.

Taken together with the base case, these two observations show that the roots of a requirement will change at most a finite number of times.\par

All that remains is to show that every requirement is satisfied.  We fix a requirement, $R_{\langle h,e \rangle}$.  Since the roots of $R_{\langle h,e \rangle}$ will change at most a finite number of times, there is a stage, $s_0$, after which $\roots(\langle h,e \rangle)$ has stabilized to a set of strings such that every member of the tree not in $\terminal$ is either a prefix or suffix of some member of the set.  Hence, every infinite branch of the tree extends a root of every requirement.  We fix $\sigma \in \roots(\langle h,e \rangle)$, a permanent root, and a stage $s_1 > s_0$ such that the outcome of $\sigma$ never changes after stage $s_1$.  Either $h$ is not total, or, at stage $s_1$, a collection of finite sets $\mathfrak W(\sigma) = \{ W_{b_{0}}, \ldots , W_{b_{h(a)}} \}$ is chosen.  If $h$ is not total, then the requirement is satisfied trivially.  Assume instead that a collection $\mathfrak W(\sigma)$ has been chosen.

\textit{Case 1:}  If the true outcome of $\sigma$ is $w_i$ for some $i\in \mathbb N$, then, following the notation introduced in the strategy, at stage $s_1$ the contents of $W_{a}$ are chosen.  A selection of $\mathfrak W(\sigma)$ is accepted by $F(\sigma) = F_i$, and the members of that selection are exactly the members of $W_{a}$.  Since the members of $\mathfrak W(\sigma)$ are disjoint, $|W_{a}| = |\mathfrak W(\sigma)| = h(a)+1$.  Thus, $h$ cannot witness effective immunity of the set computed by oracle program $e$ using any oracle that is a branch of $T$ extending $\sigma$.\par

\textit{Case 2:}  If the true outcome of $\sigma$ is $\infty$, then the requirement will ensure that no selection of $\mathfrak W(\sigma)$ is accepted by any branch passing through $\sigma$ because the branches that do accept selections of $\mathfrak W(\sigma)$ will be relegated to $\avoid(\sigma)$.  Consequently, each branch must preserve some member of $\mathfrak W(\sigma)$.  In other words, $h$ cannot witness the effective immunity of the complement of the set computed by oracle program $e$ using any oracle that is a branch of $T$ extending $\sigma$.\par

Since every requirement is satisfied for each of its roots, every branch in $T$ fails to compute an effectively bi-immune set.  The tree has two siblings immediately below every node and so some branch is a DNC function, proving the theorem.
\end{proof}

%%%%%%%%%%%%%%%%%%%%%%%%%%%%%%%%%%%%%%%%%%%%%%

\section{Conclusion}

In the proof of Theorem \ref{mainResult}, we constructed a $\Sigma_1^0$, fully-branching, infinite tree, every branch of which computes no effectively bi-immune set.  With only a modest increase in complexity, we can produce an analogous tree every branch of which is a DNC function.\par

\begin{Corollary}\label{delta2}
There is a $\Delta_2^0$ tree every branch of which is a DNC function that computes no effectively bi-immune set.
\end{Corollary}

\begin{proof}
The claim follows by pruning the tree constructed in the proof of Theorem \ref{mainResult}.  At each stage $s$ of the construction, if $\phi_{e,s}(e)$ has converged, then we remove all nodes in the tree that take the value $\phi_{e,s}(e)$ on input $e$.  The tree is limit computable since a node removed from the tree will never be a member of the tree again.  Each of the remaining branches is a DNC function that computes no effectively bi-immune set.
\end{proof}

Without proof, we state a theorem due to Ku\v{c}era, which we require to prove the next corollary.

\begin{Theorem}[Ku\v{c}era, \cite{kucera}]
Every $\Delta_2^0$ DNC degree computes a promptly simple set.
\end{Theorem}

\begin{Corollary}
There is a DNC function that computes no effectively bi-immune set and bounds a promptly simple set.
\end{Corollary}

\begin{proof}
A pruned $\Delta_2^0$ tree has a $\Delta_2^0$ member.  Thus, by Corollary \ref{delta2}, there is a $\Delta_2^0$ DNC function that computes no effectively bi-immune set.  By Ku\v{c}era's Theorem, such a DNC function also bounds a promptly simple set, proving the~claim.
\end{proof}

Finally, as a counterpoint to Theorem \ref{mainResult}, observe that there is a DNC function that does compute an effectively bi-immune set.  We fix an effectively bi-immune set, $A = \{ a_0 < a_1 < a_2 < \ldots \}$, and define a DNC function, $f$, by
$$f(x) = \begin{cases}
       \langle a_0,0 \rangle & \mbox{if }\phi_x(x) \neq \langle a_0,0 \rangle,\\
       \langle a_0,1 \rangle & \mbox{otherwise.}
\end{cases}$$
Clearly, $f$ computes $A$.

\bibliographystyle{plain}

\end{document}